\documentclass[10pt,reqno,english]{smfart}
\usepackage{amscd,amsfonts,amsmath,amssymb,amstext,amsthm}
\usepackage{mathpazo}
\usepackage{hyperref}
\newtheorem{theorem}{Theorem}[section]
\newtheorem{lemma}[theorem]{Lemma}
\newtheorem{proposition}[theorem]{Proposition}
\newtheorem{corollary}[theorem]{Corollary}
\newtheorem{mainthm}[]{Theorem}
\theoremstyle{remark}
\newtheorem{remark}[theorem]{Remark}
\newtheorem*{remarks}{Remarks}

\newtheorem{definition}[theorem]{Definition}
\title[K3-surfaces with special symmetry]
{K3-surfaces with special symmetry:\\
An example of classification by Mori-reduction}
\author{Kristina Frantzen and Alan Huckleberry
  \thanks{First
    author's research supported by grants from the Studienstiftung des
    deutschen Volkes and the Deutsche Forschungsgemeinschaft.\\
    Second author's research partially supported by grants from the Deutsche
    Forschungsgemeinschaft.}
}
\address{Institut und Fakult\"at f\"ur Mathematik\\Ruhr-Universit\"at Bochum\\Germany\\kristina.frantzen@ruhr-uni-bochum.de\\ ahuck@cplx.ruhr-uni-bochum.de}
\date {February 18, 2008}
\begin{document}
\maketitle
%
%
%
%
\section {Introduction} \label {results}
An interesting class of K3-surfaces consists of those
surfaces $X$ equipped with an antisymplectic
holomorphic automorphism $\sigma :X\to X$ of order two.  Recently there 
has been substantial progress in understanding these
manifolds, in particular as desingularized 2:1 ramified covers of 
log del Pezzo surfaces (\cite {Na}), and analytic phenomena
related to their moduli (\cite {Y1,Y2}). Here we
present an approach for studying such surfaces from the
point of view of symmetry. This amounts to analyzing
the action of a centralizer $H$ of $\sigma $ in the group of holomorphic
symplectic automorphisms of the K3-surface $X$ and the $H$-equivariant Mori-reduction 
of the quotient $X/\sigma $.

After presenting methods which apply in general, we turn  
to a special case where $H$ is the nontrivial semidirect product
$C_3\ltimes C_7$ of cyclic groups of order three and seven. 
This arose naturally in our consideration
of maximal groups of symplectic transformations, exemplifies
the general approach but requires
only minimal technical work.  Furthermore,
the results in this case shed new light on the classification of K3-surfaces
with an action of the group $L_2(7)$, which
has been studied extensively in \cite {OZ}. (See \cite {Z} for
a summary of other recent works in this direction.) Analogous 
classification results can be proved whenever $H$ 
is of sufficiently high order or has sufficiently complicated group structure 
(see \cite {FH}).
\subsection{Notation}\label{notation}
The general problem can be formulated as follows. Let $H$ be an abstract 
finite group, $A:= C_2= <\sigma>$ and $G = A\times H$. An 
effective holomorphic $G$-action on a K3-surface $X$ is defined
by an injective homomorphism $\alpha: G \to \mathrm{Aut}(X)$. It is
assumed that $\alpha(\sigma)$ is antisymplectic and 
$\alpha(H) \subset  \mathrm{Aut}_\text{sym}(X)$ is a group of symplectic
transformations of $X$.  Abusing notation, if the context is clear,
we do not differentiate between the abstract group or group elements
and their $\alpha $-images in $\mathrm {Aut}(X)$, e.g., 
writing $H$ for $\alpha (H)$. The set $\{(X,\alpha)\}$ of all $G$-actions on
K3-surfaces is denoted by $\mathcal{A}$.

We wish to describe (up to natural equivalence) 
the K3-surfaces which are equipped with
such actions. In precise terms, we regard actions $(X_1,\alpha _1)$ and 
$(X_2,\alpha _2)$ as being equivalent whenever there exists a 
biholomorphic map $\varphi :X_1\to X_2$ and a group automorphism $\psi \in
\mathrm{Aut}(G)$ such that 
$\alpha _2(g)(\varphi (x))=\varphi (\alpha_1(\psi(g))(x))$ for all $g\in G$
and $x\in X$.  The goal is then to describe the quotient
$\mathcal M=\mathcal A/\hspace{-1mm}\sim $ of all such actions by this equivalence
relation.
\subsection{Statement of results}
The fixed point set $\mathrm {Fix}_X(\sigma )$ of an antisymplectic 
automorphism on a K3-surface $X$ is a disjoint union of smooth curves. Unless
this set is empty, the quotient $Y:=X/\sigma $ is a smooth 
rational surface.  Operating in the setting described above, we
have the canonically induced $H$-action on $Y$.  It is then
possible to apply equivariant Mori-reduction to obtain an
$H$-equivariant map
$Y\to Y_{\mathrm {min}}$ to an $H$-minimal model which either
has an ample anticanonical bundle or is an equivariant fiber
space over $\mathbb P_1$ with general fiber $\mathbb{P}_1$. If $Z$ is
a surface with $-K_Z$ ample and $K_Z^2 =d$, then $1 \leq d \leq 9$, and
one refers to $Z$ as a del Pezzo surface of degree $d$. Such a
surface is either the blow up of $\mathbb{P}_2$ in $9-d$ points in
general position or $\mathbb{P}_1 \times \mathbb{P}_1$.

In \S 2, we study the exceptional curves of the equivariant Mori
reduction $Y \to Y_{\mathrm{min}}$, in particular their 
intersection with the branch locus of the
quotient map $\pi: X \to X/\sigma = Y$. We exemplify our 
method by classifying $Y_\text{min}$ and $Y$ for $H$ being the unique nontrivial semidirect product
$ C_3\ltimes C_7$ defined by the
action of $ C_3$ on $ C_7$ which is given by its embedding
in $ C_6\cong \mathrm {Aut}(C_7)$.  Having done so, the classification problem for $X$ reduces to a study of $H$-invariant sextics in $\mathbb{P}_2$ and
we are able to describe the space $\mathcal M$ by 
direct calculation (see Section \ref{computation}).

Up to $\mathrm {SL}_3(\mathbb C)$-conjugation there are two effective
actions of $H$ on $\mathbb P_2$. These are equivalent in the above sense as they differ by an outer
automorphism of $H$.
A double cover of $\mathbb{P}_2$ branched along a smooth $H$-invariant curve of degree six is a K3-surface with an action of $H$.
Those branch curves which define K3-surfaces with a
symplectic action of $H$ are defined by $H$-invariant polynomials. 
Choosing coordinates,
the space $\mathbb C[z_0,z_1,z_2]^H_{(6)}$ of $H$-invariant polynomials
of degree six is
$$
V:=\mathrm {Span}\{z_0^2z_1^2z_2^2,\, z_0^5z_1+z_2^5z_0+z_1^5z_2\}\,.
$$
The family $\mathbb{P}(V)$ of curves defined by polynomials in $V$
contains exactly four singular curves. These are the curve defined by
$z_0^2z_1^2z_2^2$ and those defined by 
$3z_0^2z_1^2z_2^2 -\zeta ^k(z_0^5z_1+z_2^5z_0+z_1^5z_2)$, where
$\zeta $ is a nontrivial cube root of unity, $k=1,2,3$.

To define the equivalence relation which then yields a description
of $\mathcal M$, we let $\Gamma$ be the cyclic group of order
three generated by the transformation 
$[z_0:z_1:z_2]\mapsto [z_0:\zeta z_1: \zeta ^2z_2]$, where $\zeta $
is as above. This group acts on the space $\mathbb P(V)$ of
$H$-invariant curves and one shows that two such curves are
equivalent by means of an automorphism of $\mathbb P_2$
if and only if they lie in the same $\Gamma$-orbit. In particular, the three irreducible singular $H$-invariant curves form a $\Gamma$-orbit.

The singular curve $C_\text{sing} \subset \mathbb P_2$ defined by
$3z_0^2z_1^2z_2^2 -(z_0^5z_1+z_2^5z_0+z_1^5z_2)$ has exactly seven
singular points $p_1, \dots p_7$ forming an $H$-orbit. Since they are in
general position, the blow up of $\mathbb{P}_2$ in these points defines a del Pezzo surface $Y_\text{Klein}$ of degree two with an action of $H$ and is seen to be the double cover of $\mathbb{P}_2$ branched along Klein's quartic curve 
$$
C_\text{Klein}:=\{z_0z_1^3+z_1z_2^3+z_2z_0^3=0\}.
$$
The proper transform $B$ of $C_\text{sing}$ in $Y_\text{Klein}$ is a
smooth $H$-invariant curve of genus three and coincides with the
preimage of $C_{\text{Klein}}$ in $Y_\text{Klein}$.
The minimal resolution $\tilde X_\text{sing}$ of the singular surface $X_\text{sing}$ defined as the double cover of $\mathbb P_2$ branched along $C_\text{sing}$ is a K3-surface with an action $H$. 
 By construction, it is the double cover of
$Y_\text{Klein}$ branched along $B$. In particular, $\tilde X
_\text{sing}$ is the degree four cyclic cover of $\mathbb P_2$ branched along $C_\text{Klein}$ and known as the Klein-Mukai-surface $X_{\text{KM}}$. 
The Klein-Mukai-surface is the unique K3-surface in the family $\mathcal M$ such that $Y \not\cong \mathbb{P}_2$.

Thus,
letting $\Sigma $ be $\mathbb P(V)$ with the reducible curve $\{z_0^2z_1^2z_2^2=0\}$
removed, the parameter space $\mathcal M$ is given by $\Sigma /\Gamma $ and the
description of K3-surfaces with $G$-actions can be formulated as follows.
\begin {mainthm}\label {main theorem}
The K3-surfaces equipped with an action of 
$H = C_3 \ltimes C_7$ of holomorphic symplectic 
transformations which is centralized by an antisymplectic 
involution $\sigma$  are parameterized
by the space $\mathcal M=\Sigma /\Gamma $ of equivalence classes of sextic branch curves in $\mathbb{P}_2$. The Klein-Mukai-surface occurs as the minimal desingularization of the double cover branched along the unique singular curve in $\mathcal M$.
\end {mainthm}
\subsection{Application to the case $H = L_2(7)$}
In \cite {M} Mukai gives a classification of abstract finite
groups which can occur as groups of symplectic transformations
of a K3-surface. Precisely, he presents a list of eleven finite groups
such that a finite group 
occurs as a group of symplectic transformations on a K3-surface if and
only if it occurs as a subgroup of one of the groups in this list. For
every entry $H$ in his list, Mukai gives an explicit
example of a K3-surface with a symplectic $H$-action, but these examples are
by no means exhaustive. It is therefore of interest to describe all
surfaces where the groups from this list occur.

One first approach to understanding this situation is to consider
finite groups $G$ acting on K3-surfaces with normal subgroups
$H$ of symplectic transformations which are maximal in the
above sense, i.e., which appear in Mukai's list. In other
  words, if we consider the $G$-action, $\omega \mapsto \chi (g)\cdot \omega $,
on any symplectic form, then $H$ is the kernel of the character $\chi $
and $\chi $ identifies $G/H$ with some cyclic group $C_k$.

A particular example of a group on Mukai's list is $H=L_2(7)$. In this
case, we may identify $H$ with
its group of inner automorphisms and consequently the
map $G\to \mathrm {Aut}(H)$ defined by conjugation induces a
homomorphism 
$$
C_k\cong G/H\to \mathrm {Out}(H)\cong  C_2\,.
$$
Therefore, except in the case where $G$ is the only nontrivial
semidirect product $G= C_2\ltimes L_2(7)$, if $k\geq 2$, then
$H$ is centralized by a cyclic subgroup $C_m$ of nonsymplectic
transformations.  In fact one can show that $m=2,4$ are the only
possibilities (cf. \cite {OZ}; for a proof 
using Mori-reduction see \cite {F}). In both cases the existence of an
antisymplectic involution $\sigma$ centralizing $H$ follows.
\begin{remark}
Studying actions of a finite group $G$ on K3-surfaces, one finds in
many cases a cyclic group of nonsymplectic transformations as above 
which centralizes at least 
an interesting subgroup of the group of symplectic transformations in
$G$. This is the general principle that led to our interest in this subject.
\end{remark}
Now let us return to our original notation with 
$H= C_3\ltimes C_7$, but here regarded as a 
subgroup of $L_2(7)$. If an involution $\sigma$ centralizes $L_2(7)$,
then it certainly centralizes $H$. Consequently, 
$\mathcal M(L_2(7))$ is contained in $\mathcal M(H)$. 
The following is proved by checking which elements of $\mathcal M(H)$ have the symmetry of
the larger group $L_2(7)$.
\begin {mainthm} \label{thmL2(7)}
Among the K3-surfaces having a symplectic action of 
$C_3\ltimes C_7$ centralized by an antisymplectic
involution there are exactly two which
are equipped with $L_2(7)$-actions centralized by the same
involution.  These are $X_{\mathrm {KM}}$ and the surface 
defined as the 2:1 cover of $\mathbb P_2$ ramified over
the curve 
$\mathrm{Hess}(C_\text{Klein})=\{z_0^5z_1+z_2^5z_0+z_1^5z_2-5z_0^2 z_1^2 z_2^2=0\}$ .
\end {mainthm}
\begin{remark}
In general,
the $H$-minimal model might not be that of a larger group. Here this 
occurs in the case of $X_{\mathrm{KM}} $, where $X_{\mathrm{KM}}/\sigma$ 
is $L_2(7)$-minimal but is not $H$-minimal.
\end{remark}
The following result of Oguiso and Zhang (\cite{OZ}) is a 
consequence of this theorem.
\begin{corollary}\label{corL2(7)}
If $X$ is a K3-surface with an action of a finite group $G$ containing 
$L_2(7)$ such that $|G / L_2(7)|=4$, then $X$ is the Klein-Mukai surface.
\end{corollary}
\begin{proof}[Proof of Corollary \ref{corL2(7)}]
Since $L_2(7)$ is simple and maximal in the above sense, 
it coincides with the group of symplectic transformations in $G$. 
In particular, $G / L_2(7) = C_4$ and a group 
$\langle \sigma \rangle$ of
order two is contained in the kernel of $G\to \mathrm {Aut}(L_2(7))$.
Consequently we are in the setting of Theorem \ref{thmL2(7)} where
$\Lambda :=G/\langle \sigma \rangle$ acts on $Y=X/\sigma $.  If
$X\not =X_{KM}$, then $Y= \mathbb P_2$. To complete the proof
we must eliminate this possibility.

For this let $\tau $ be any element of $\Lambda $ which
is not in $L_2(7)$ and consider the conjugate $\tau H\tau ^{-1}$.
Since any two subgroups of order 21 in $L_2(7)$ are conjugate to
each other by an element of $L_2(7)$, it follows that there exists
$g\in L_2(7)$ with $(g\tau )H(g\tau)^{-1}=H$.  Thus the normalizer
$N(H)$ in $\Lambda $ is a group of order 42 which also normalizes the commutator subgroup
$H'$ and therefore stabilizes its set $F$ of fixed points.

Using the same coordinates $(z_0:z_1:z_2)$ of $\mathbb{P}_2$ 
as in Theorem \ref{thmL2(7)} one directly
checks that the only transformations in $\mathrm {Stab}(F)$ which stabilize 
the branch curve $\mathrm {Hess}(C_{\mathrm {Klein}})$
are those in $H$ itself.  This contradiction shows that
$Y\not =\mathbb P_2$ and therefore $X=X_{KM}$. 
\end{proof}
\begin {remark}
The assumptions of the corollary may in fact we weakened. As we remarked above, one can show that $|G / L_2(7)|\geq 3$ implies $|G / L_2(7)|=4$ (see \cite{OZ} or \cite{F}).
\end {remark}
Let us conclude this introduction with a brief outline of
this note. In $\S\ref{general methods}$ 
we explain the equivariant Mori-reduction and prove several general facts about 
the position of the Mori-fibers with respect to the images in $Y$ 
of the $\sigma $-fixed point curves.  These
will be used in $\S\ref{fine classification}$ to give the proofs 
of the theorems
which are stated above. They will also play a fundamental role
in \cite {FH} where we prove analogous theorems for any
group $H$ which is either sufficiently large or sufficiently
complicated.
%
%
%
%
\section {General Methods}\label {general methods}
\subsection{Quotients of K3-surfaces by antisymplectic involutions}
Continuing with the notation of Section \ref{notation}, where
$H$ is an arbitrary finite group of symplectic automorphisms
on a K3-surface $X$ and $\sigma $ is an antisymplectic
involution which centralizes it, we begin
by recalling that there are strong restrictions on the nature
of $\mathrm {Fix}(\sigma )$ (\cite {AN}).
\begin {proposition} \label {AN-description}
If $\sigma$ is an antisymplectic involution on a K3-surface
$X$, then $\mathrm{Fix}_X(\sigma)$ is one of the following types:
\[
1.)\ \mathrm{Fix}_X(\sigma) = D_g \cup \bigcup_{i=1}^n R_i,
\quad \quad
2.)\ \mathrm{Fix}_X(\sigma) = D_1^{(1)} \cup D_1^{(2)},
\quad ¸\quad
3.)\ \mathrm{Fix}_X(\sigma)= \emptyset,
\]
 where $D_g$ denotes a smooth curve of genus $g$ and $\bigcup_{i=1}^n R_i$
 is a possibly empty union of smooth disjoint rational curves. 
\end {proposition}
To clarify the notation, this means in particular that if there are
at least two elliptic curves in $\mathrm{Fix}_X(\sigma)$, then the
$\sigma$-fixed point set is the union of two elliptic curves, i.e., case 2.) occurs. It
also should be remarked that the total number of curves in
$\mathrm{Fix}_X(\sigma)$ is at most 10 (\cite {Z2}). Using this fact
it would be possible to shorten one of our combinatorial arguments in the proof
of Proposition \ref{n=0}, but instead we give a self-contained presentation.

Let $\pi: X \to X/ \sigma = Y$ denote the quotient map. In the following, we
assume that $\mathrm{Fix}_X(\sigma)$ is nonempty. This implies that 
the quotient surface $Y$ is a smooth rational surface with an action of $H$. 
\begin{remark}
Note that the centralizer of an automorphism stabilizes its set of
fixed points.  If $h$ is a symplectic automorphism of order 7 on a 
K3-surface $X$, then $\vert \mathrm {Fix}_X(h)\vert =3$ (\cite{N}). 
In particular, if an automorphism $\sigma $ of order two centralizes $h$, then 
$\mathrm {Fix}_X(\sigma)\cap \mathrm {Fix}_X(h)\not=\emptyset$.
\end{remark}
\subsection{The equivariant minimal model program for surfaces}
Mori's minimal model program may be adjusted to provide an
equivariant version for projective varieties equipped with actions
of finite groups (Example 2.18 in \cite{KM}).  Details
in the case of interest in the present note, i.e., for smooth surfaces,
are provided in \cite {F}.
Here we give a brief description of this $\emph{Mori-reduction}$.

Let $Y$ be a smooth projective surface with an action of a finite
group $H$. Equivariant analogues of the cone and contraction
theorems provide the following classification
result. 
\begin{proposition}
There exists a sequence of $H$-equivariant
extremal contractions $Y \to Y_{(1)} \to \dots \to Y_{\mathrm{min}}$ such that
$Y_{\mathrm{min}}$ satisfies one of the following conditions:
\begin{enumerate}
\item
$K_{Y_\mathrm{min}}$ is nef;
\item
$Y_\mathrm{min}$ is an $H$-equivariant conic bundle over a smooth curve,
i.e., there exists an $H$-equivariant morphism
$Y_{\mathrm{min}} \to C$ onto a smooth curve $C$ such that the general
fiber is a rational curve;
\item
$-K_{Y_{\mathrm{min}}}$ is ample.
\end{enumerate}
Each extremal contraction is the contraction of an $H$-orbit of disjoint $(-1)$-curves.
\end{proposition}
\begin{definition}
The surface $Y_\text{min}$ is referred to as an  \emph{$H$-minimal model of $Y$},
and the map $Y \to Y_\text{min}$ is called a \emph{Mori-reduction}. 
A (connected) curve $E \subset Y$ is called \emph{Mori-fiber} 
if it is contracted
in some step of the Mori-reduction.
The set of all Mori-fibers is denoted by $\mathcal{E}_{\text{Mori}}$.
\end{definition}
In the present note we apply the equivariant minimal model program
to the rational surface $Y$ obtained as a quotient of the K3-surface 
$X$ by an antisymplectic involution $\sigma $. Here $Y$ is equipped 
with the action of the finite group $H$ of holomorphic automorphisms
which initially was acting on $X$ and centralized by $\sigma $.
An $H$-minimal model of $Y$ can either be a del Pezzo surface 
or an equivariant conic bundle over $\mathbb{P}_1$. We let
$n$ denote the number of rational curves in $\mathrm{Fix} (\sigma)$.
\begin{remark}
The Euler characteristic of the branched double cover $\pi: X \to Y$
is computed as
$e(X)=24=2e(Y) - \sum e(C)$,
where the sum is taken over all connected components $C$ of $\mathrm{Fix}_X(\sigma)$. Hence
$$
24\geq 2e(Y) - 2n \notag =2e(Y_\mathrm{min}) + 2|\mathcal{E}_{\text{Mori}}| -2n \geq 6 + 2|\mathcal{E}_{\text{Mori}}| -2n
$$
 and the total number of Mori-fibers $|\mathcal{E}_{\text{Mori}}|$ is bounded by $n+9$.
\end{remark}
\subsection{Branch curves and Mori-fibers}\label{branch curves and mori-fibers}
Let $R:= \mathrm{Fix}_X(\sigma) \subset X$ denote the ramification
locus of $\pi$ and let $B=\pi(R) \subset Y$ be its branch locus. We
denote by $\mathcal{E} \subset \mathcal{E}_{\text{Mori}}$ the set of
all Mori-fibers which are not contained in the branch locus $B$.
\begin{lemma}\label{preimage of E in X}
Let $E \in \mathcal{E}$ be a Mori-fiber such that $|E\cap B| \geq 2$ or $(E,B) \geq 3$. Then $\pi^{-1}(E)$
is a smooth rational curve in $X$. 
\end{lemma}
\begin{proof}
Let $k < 0$ denote self-intersection number of $E$. The divisor
$\pi^{-1}(E)$ has self-inter\-section $2k$. Assume that $\pi^{-1}(E)$ is
reducible and let $\tilde E_1, \tilde E_2$ denote its irreducible components.  They are rational
and therefore have self-intersection number $-2$.  Write
$
2k = (\pi^{-1}(E))^2 = \tilde E_1^2 + \tilde E_2^2 + 2 (\tilde E_1,\tilde E_2)
$. 
Since $\tilde E_1$ and $\tilde E_2$ intersect at points in the preimage of $E \cap B$, we obtain $(\tilde E_1,\tilde E_2) \geq 2$, a contradiction. It follows that $\pi^{-1}(E)$ is irreducible. Consequently, $k=-1$ and $\pi^{-1}(E)$ is a smooth rational curve.
\end{proof}
\begin{remark}\label{self-int of Mori-fibers}
Considering the preimage $\pi ^{-1}(E) \subset X$ of a Mori-fiber $E
\in \mathcal E$ it follows from adjunction on $X$ that $E$ is a $(-1)$-curve if
and only if $E \cap B \neq \emptyset$; all Mori-fibers disjoint from
$B$ are $(-2)$-curves and all Mori-fibers $E \subset B$ are $(-4)$-curves.
In particular, if $E_1,E_2\in \mathcal E$ are two Mori-fibers which
have nonempty intersection, then $E_1\cap E_2$ is contained in the
complement of $B$.
\end{remark}
\begin {proposition} \label {at most two}
Every Mori-fiber $E \in \mathcal{E}$ meets the branch locus $B$ in at
most two points. If $E$ and $B$ are tangent at $p$, then 
$E\cap B = \{p\}$ and $(E,B)_p =2$.
\end {proposition}
\begin{proof} 
Let $E \in \mathcal E$ and assume $|E\cap B| \geq 2$ or $(E,B) \geq
3$. Then by the lemma above, $\tilde E = \pi^{-1}(E)$ is a smooth
rational curve in $X$. Since $\tilde E \not\subset \mathrm{Fix}_X(\sigma)$, the involution $\sigma$ has exactly two fixed points on $\tilde E$ showing $|E\cap B| = 2$\\ 
Let $N_{\tilde{E}}$ denote the normal bundle of $\tilde{E}$ in $X$. We
consider the induced action of $\sigma$ on $N_{\tilde{E}}$ by a bundle
automorphism. Using an
equivariant tubular neighborhood theorem we may equivariantly identify a
neighborhood of $\tilde E$ in $X$ with $N_{\tilde E}$ via a 
$C^{\infty}$-diffeomorphism. The $\sigma$-fixed point curves
intersecting $\tilde{E}$ map to curves of $\sigma$-fixed points in
$N_{\tilde{E}}$ intersecting the zero-section and vice versa. 
Let $D$ be a curve of $\sigma$-fixed point in $N_{\tilde{E}}$. If $D$ is
not a fiber of $N_{\tilde E}$, it follows that $\sigma$ stabilizes all 
fibers intersecting
$D$ and the induced action of $\sigma$ on the base must be trivial, a contradiction.
It follows that the $\sigma$-fixed point curves correspond to fibers of 
$N_{\tilde{E}}$, and $E$ and $B$ meet transversally.\\
In particular, if $E$ and $B$ are tangent at $p$, then $|E \cap B |=1$ and $(E,B)=2$.
\end{proof}
%
%
%
\section {Fine classification} \label {fine classification}
For the remainder of this paper $H:= C_3\ltimes  C_7$, 
$A:=C_2=\langle \sigma  \rangle $ and 
$G=A\times H$ is acting as in $\S\ref{notation}$ on a K3-surface $X$. 
Here we prove the two theorems formulated in $\S\ref{results}$.
\subsection {Representation as ramified cover}
Recall that the commutator subgroup $H'\cong C_7$ has exactly three
fixed points in $X$ (\cite {N}). Since $H$ has no faithful 
2-dimensional representation, it must act transitively on
$\mathrm {Fix}_X(H')$. Now $\sigma$ stabilizes $\mathrm {Fix}_X(H')$ as well
and, as was mentioned above, has at least one fixed point there.
However, since $H$ acts transitively on $\mathrm {Fix}_X(H')$ and 
centralizes $\sigma $,
it follows that $\sigma $ fixes $\mathrm {Fix}_X(H')$ pointwise,
i.e., $\mathrm {Fix}_X(H')\subset \mathrm {Fix}_X(\sigma )$ and in particular
$\mathrm {Fix}_Y(H')$ also consists of three points $y_1,y_2,y_3$.

We now turn to an analysis of the branch curves of the 
covering $\pi: X \to X/\sigma =Y$.
\begin {proposition}
The set $\mathrm {Fix}_Y(H')$ is contained in a unique (connected) branch
curve $B_0$.
Its genus $g=g(B_0)$ is at least three.
\end {proposition}
\begin {proof}
Let $B_i$ be the connected branch curve which contains $y_i$ and observe
that $H$ acts transitively on the set $\mathcal B:=\{B_1,B_2,B_3\}$.
Since it is acting as $C_3$, the set $\mathcal B$ does not
consist of two elements.  If $\vert \mathcal B\vert =3$, then
from Proposition \ref{AN-description} it follows that at
least two of the $B_i$ must be rational.  Since each $B_i$ is
stabilized by $H'$, it then follows that $H'$ has at least two
additional fixed points contradicting $|\mathrm{Fix}_Y(H')|=3$.  
Therefore $\mathcal B$ consists of
one curve which we denote by $B_0$.

Now $B_0$ is $H$-invariant and, since it is acting as a group
of symplectic transformations, $H$ is acting effectively there.
Furthermore, $H'$ has exactly three fixed points in $B_0$.
Hence $B_0$ is not rational.  There is indeed an elliptic curve
with an effective $H$-action, but in such a case $H'$
must act by translations, i.e., fixed point free.  Thus $B_0$
is not elliptic.  If $g(B_0)=2$, then $B_0$ is hyperelliptic.
But since the quotient of $B_0\to \mathbb P_1$ by the hyperelliptic 
involution is $\mathrm {Aut}(B_0)$-equivariant and there is no effective
action of $H$ on $\mathbb P_1$, this is also not possible.
Consequently $g=g(B_0)\ge 3$ as claimed.
\end {proof}
Let $\{B_1,\ldots ,B_n\}$ be the set of image
curves which do not contain an $H'$-fixed point, i.e., $B_i\not=B_0$
for all $i$.  Again using Proposition \ref{AN-description} we
see that every $B_i$ is rational.  Thus, since $H'$ has no
fixed points in $\cup B_i$ for $i>0$, we observe that $H'$ is acting
freely on $\{B_1,\ldots ,B_n\}=\mathcal R$.  

Let us now first eliminate the possibility that an $H$-minimal model
of $Y$ is a $\mathbb P_1$-fiber space. 
\begin{lemma}
An $H$-minimal model of $Y$ is a del Pezzo surface.
\end{lemma}
\begin{proof}
It is necessary to exclude the case of $Y_\mathrm{min}$ being an
$H$-equivariant conic bundle
$Y_\mathrm{min}\to \mathbb{P}_1$. 
Since there is no effective action of $H$ on $\mathbb P_1$, it follows
that the only proper normal subgroup $H'$ of $H$ acts trivially
on the base.  But the generic fibers 
are isomorphic to $\mathbb P_1$ and $H'$ must have at least 
two fixed points in each such fiber, contrary to it
having only three fixed points in $Y$. 
\end{proof}
\begin {proposition}
For every $B_i\in \mathcal R$ it follows that 
$\vert H.B_i\vert =7$ and there are at most two
such $H$-orbits in $\mathcal R$. In particular,
 $n\in \{0,7,14\}$.
\end {proposition}
\begin {proof}
Since $\mathrm{Pic}(X)$ does not admit a 
negative-definite sublattice of rank $21$ and $H'$ stabilizes no curve in $\mathcal R$, 
the only possibility is $\vert H.B_i\vert =7$ and $n \in \{0,7,14\}$. 
\end {proof}
The following is the main result of this section.
\begin {proposition}\label{n=0}
There are no branch curves other than the one which
contains $\mathrm {Fix}_Y(H')$, i.e., $n=0$ and $B = B_0$.
\end {proposition}
Since the proof requires several combinatorial arguments, for
the sake of clarity we separate it into a number of steps.  Throughout
we assume that $n\not=0$ and at the end reach a contradiction.
This assumption implies that $\mathcal R$ consists of either
one or two $H$-orbits consisting of seven curves and
$n=7,14$, respectively. We make extensive use of the results in section \ref {branch curves and mori-fibers} and the fact that an irreducible curve on a del Pezzo surface has self-intersection $\geq -1$. The desired contradiction will be
that $\vert \mathcal {E}_{\mathrm {Mori}}\vert $ is larger than
the estimate guaranteed by the following observation.
\begin {lemma}\label{estimate}
If $n=7$, then $\vert \mathcal {E}_{\mathrm {Mori}}\vert $
is at most $13$. If $n=14$ , then $\vert \mathcal {E}_{\mathrm {Mori}}\vert $
is at most $20$.
\end {lemma}
\begin {proof}
Assume that $\vert \mathcal {E}_{\mathrm {Mori}}\vert $ is larger
than the claim.  Then,
using the Euler characteristic formula 
\begin {equation} \label{Euler characteristic formula}
13-g(B_0)=e(Y_\mathrm{min})+\vert \mathcal E _{\text{Mori}} \vert -n
\end {equation}
along with $g(B_0) \ge 3$ and $e(Y_\mathrm{min})\ge 3$, we see that $g=3$, $n=7$ implies
$\vert \mathcal E_{\text{Mori}} \vert = 14$, $n=14$ implies
$\vert \mathcal E_{\text{Mori}} \vert = 21$, and $e(Y_\mathrm{min})=3$. 
Consequently, $Y_{\mathrm{min}}=\mathbb P_2$. The specified values of
$n$ and  $\vert \mathcal E_{\text{Mori}} \vert $ guarantee $\mathcal E
_{\text{Mori}} \cap \mathcal R =\mathcal E
_{\text{Mori}} \backslash \mathcal E = \emptyset$. To see this,
e.g., in the case where $n=7$, note that if one curve $B_i$ is
contracted to a point by the reduction $Y\to Y_\mathrm{min}$, then all
curves in $\mathcal R$ are contracted and $|\mathcal E
_{\text{Mori}} \backslash \mathcal E| =7$. Since $B_i^2 = -4$, more than seven additional
Mori-fibers are required in order to come to a step in the
Mori-reduction where the curves $B_i$ can be contracted.
So all branch curves are mapped to curves in $Y_\mathrm{min}=\mathbb P_2$.
However, using Remark \ref{self-int of Mori-fibers}
one sees that in this 
situation there is no configuration of Mori-fibers  
such that the images
in $Y_{\mathrm{min}}=\mathbb P_2$ of every pair of branch curves 
have nonempty intersection.  
\end {proof}
Now let $B_i$ be any branch curve in $\mathcal R$ 
and $I_{B_i}:=\mathrm {Stab}_H(B_i)$ It follows that $I_{B_i} \cong C_3$.
\begin {lemma}\label{B_i}
Every Mori-fiber which meets $B_i$ is $I_{B_i}$-invariant 
and meets no other branch curve
in the orbit $H.B_i$.  Furthermore, $B_i$
meets at most two Mori-fibers in $\mathcal E$.
\end {lemma}
\begin {proof} 
Suppose there is an orbit $I_{B_i}.E_1=\{E_1,E_2,E_3\}$ consisting of 
three different Mori-fibers which
meet $B_i$. Let $S:=\mathrm {Stab}_H(E_1)$. It follows from Lemma \ref{estimate} above 
that $S$ is nontrivial. If $S$ does not stabilize $B_i$, then $\vert S.B_i\vert \ge 3$,
contrary to $\vert E_1\cap B\vert \le 2$ (cf. Proposition \ref{at most two}).  Thus
$S=I_{B_i}$ and $I_{B_i}.E_1 = \{E_1\}$. If $E_1$ meets any other
branch curve in $H.B_i$, then it meets at least three
others, again contrary to $\vert E_1\cap B\vert \le 2$.
Finally, if $B_i$ meets more than two Mori-fibers in $\mathcal E$, then at
least one Mori-fiber $E$ meets $B_i$ outside
$\mathrm{Fix}_{B_i}(I_{B_i})$. Since $E$ is $I_{B_i}$-invariant, it follows that $E \cap B$ consists of more than three points, a contradiction.
\end {proof}
\begin {lemma}
Every $B_i\in \mathcal R$ meets exactly one Mori-fiber
in $\mathcal E$.
\end {lemma}
\begin {proof}
First note that a $(-4)$-curve in $Y$ cannot be mapped biholomorphically to the del Pezzo surface $Y_\mathrm{min}$ and therefore every curve $B_i \in \mathcal R$ meets at least one
Mori-fiber.
Assume some $B_i$ meets at least three Mori-fibers.
Since none of these meets another branch curve in
the orbit $H.B_i$ by Lemma \ref{B_i}, we obtain $\vert \mathcal E_{\mathrm {Mori}}\vert \ge 21$,
contrary to Lemma \ref{estimate}.  If $n=7$, then the
same argument proves the desired result. It remains to consider the
case $n=14$.
If $B_i$ meets
two Mori-fibers $E_1,E_2$ and both intersections are transversal,
then, since the image of $B_i$ in the del Pezzo surface
$Y_{\mathrm {min}}$ cannot be a $(-2)$-curve, at least one of
these two Mori-fibers meets a third one $E_3$. Since $E_1^2=E_2^2=-1$ and $E_3^2 \leq -2$ (cf. Remark \ref{self-int of Mori-fibers}) this Mori-fiber $E_3$ is
not among $H.E_1$ or $H.E_2$ and by Lemma \ref{B_i} the full configuration consists of at
least 21 Mori-fibers. If $B_i$ meets two Mori-fibers $E_1,E_2$ and
$E_1$ is tangent to $B_i$, then Proposition \ref{at most two} ensures that $E_1$ meets no other curve in $\mathcal
R$. Even if $E_2$ meets a branch curve $B_k$ in the other $H$-orbit $\mathcal R \backslash (H.B_i)$ additional Mori-fibers meeting $H.B_k$ are required and again the total number of Mori-fibers exceeds 21.
\end {proof}
We have now reduced to the situation where every branch
curve $B_i \in \mathcal R$ meets exactly one Mori-fiber $E\in \mathcal E$. This final
case requires a closer look at the intersection diagram of branch
curves and Mori-fibers: each $B_i \in \mathcal R$ fulfills one of
the following possibilities:
\begin{enumerate}
\item
$E \cap B_i = \{p_1,p_2\}$ or
\item
$E \cap B_i = \{p\}$ and $ (E, B_i)_p =2$ or
\item
$E \cap B_i = \{p\}$ and $ (E, B_i)_p =1$.
\end{enumerate}
\begin{proof}[Proof of Proposition \ref{n=0}]
In cases 1.) and 2.), by Proposition \ref{at most two} the only branch curve which is met by 
$E$ is $B_i$ itself. The
blowing down of $E$ transforms $B_i$ into a singular curve of
self-intersection zero. Since a del Pezzo surface does not admit a
curve of this type, there exists a Mori-fiber $E_1 \in \mathcal E$ with 
$E_1 \cap E \neq \emptyset$. By Remark \ref{self-int of Mori-fibers}
the Mori-fiber $E_1$ does not meet the branch locus $B$ and $E_1^2=-2$.
Furthermore, $E_1$ meets no $(-1)$-curve among the
Mori-fibers except $E$. Thus we have found 14 Mori-fibers along the orbit
$H.B_i$. If $n=7$, this yields the desired contradiction.
If $n=14$, then none of these 14 Mori-fibers meets a
branch curve in the other $H$-orbit, which must therefore
contribute at least 7 additional Mori-fibers. The total number of
Mori-fibers in this configuration violates Lemma \ref{estimate}.

It remains to consider the situation where all intersections of $B_i
\in \mathcal R$ with Mori-fibers are as in case 3.).
As above, we deduce the
existence of a Mori-fiber $E_1$ which meets $E$ in exactly one
point. If $n=7$, we have reached the desired contradiction. Hence, we
may suppose that $n=14$. 

Let us check that without loss of generality we can assume that $E_1 \in
\mathcal E$, i.e., $E_1 \not\in \mathcal R$. 
If $E_1 \in \mathcal R$, i.e., $E_1 = B_k$ for some $B_k \in \mathcal R \backslash (H.B_i)$, then blowing down $E$ transforms $B_i \cup B_k$ into a pair of intersecting $(-3)$-curves. It follows that there exists a Mori-fiber $E_2 \in \mathcal E$ with $E_2 \cap E \neq \emptyset$ which we then pick instead of $E_1$.

So let $E_1 \in \mathcal E = \mathcal E _\mathrm{Mori} \backslash \mathcal R$.
In particular, $E_1$ has
self-intersection $-2$, meets no branch curve, no $(-1)$-curve
among the Mori-fibers except $E$, and meets $E$ in exactly one point. It follows that $E_1$ must be
$I_{B_i}$-invariant, since otherwise we find at least 28
Mori-fibers along $H.B_i$. 

If $E$ meets a branch curve $B_k$ 
in the other $H$-orbit, then $B_k$ must also be 
$I_{B_i}$-invariant, as otherwise $E$ would meet at least
four branch curves. We obtain a contradiction since the
$I_{B_i}$-action on $E$ may not fix the three points of intersection
of $E$ with $B_i$, $B_k$ and $E_1$.  Consequently, $E$ cannot meet
a branch curve in the other $H$-orbit, which therefore contributes at least seven Mori-fibers, and $|\mathcal E_{\text{Mori}}|
\geq 21$.
\end{proof}
If $X/\sigma =Y$ is not $H$-minimal, then it is characterized by the
following observation.
\begin {proposition}
The surface $Y$ is either $H$-minimal or the blow up of $\mathbb P_2$ in seven singularities of an irreducible $H$-invariant sextic..
\end {proposition}
\begin {proof}
Since $n=0$, the Euler characteristic formula 
(\ref {Euler characteristic formula}) yields 
$\vert \mathcal E\vert \le 7$.  The fact that
$H$ acts on $\mathcal E$ then implies that 
$\vert \mathcal E\vert \in \{ 0,3,6, 7\}$.  If $\vert \mathcal E\vert \in \{3, 6\}$,
then $H'$ stabilizes every $E\in \mathcal E$, and
consequently it has more then three fixed points, a contradiction. Thus we must only 
consider the case $\vert \mathcal E\vert =7$.

Since $\mathcal E$ is an $H$-orbit, it follows that every
$E\in \mathcal E$ has self-intersection $-1$ and therefore has
nonempty intersection with $B_0$ by Remark \ref{self-int of Mori-fibers}.

The Euler characteristic formula
again implies that $g(B_0)=3$ and $Y_\mathrm{min}=\mathbb P_2$ and 
adjunction in $X$ shows that $B_0.B_0=8$ in $Y$. The fact that
$B_0$ has nonempty intersection with seven different Mori-fibers implies
that its image $C$ in $Y_\mathrm{min}$ has self-intersection either $15 = 8 +7$ or $36 = 8 + 4 \cdot 7$. Since the first is impossible it follows that $(E,B_0)=2$ for all $E \in \mathcal E$ and the $H$-invariant irreducible sextic $C$ has seven singular points corresponding to the images of $E$ in $\mathbb P_2$.
\end {proof}
\begin{corollary}\label{rough classi}
If $Y$ is not $H$-minimal, then $X$ is the minimal desingularization of a double cover of $\mathbb{P}_2$ branched along an irreducible $H$-invariant sextic with seven singular points.
\end{corollary}
We conclude this subsection with a classification of possible
$H$-minimal models of $Y$.
\begin {proposition}
The surface $Y_\mathrm{min}$ is either a del Pezzo surface of degree two
or $\mathbb P_2$.
\end {proposition}
\begin {proof}
If $Y_\mathrm{min}=\mathbb P_1\times \mathbb P_1$, then, since there are
no nontrivial homomorphisms of $H$ to $C_2$, it
follows that $H$ is contained in the connected component
of $\mathrm {Aut}(Y_\mathrm{min})$.  But this is also not possible, because
there are no injective homomorphisms of $H$ into 
$\mathrm {PSL}_2(\mathbb C)$. Thus $Y_\mathrm{min}=Y_d$ 
is a del Pezzo surface of degree $d=1,\ldots ,9$ which is a blowup
of $\mathbb P_2$ in $9-d$ points.

It is immediate that $d\not=1$, because the anticanonical map of
such a surface has exactly one base point.  
Since this would have to be $H$-fixed and $H$ has no faithful 
2-dimensional representations, this case does not occur and
we must only eliminate $d=3,\ldots ,8$.  However, in these
cases the sets $\mathcal K$ of (-1)-curves consist, respectively,
of $27,16,10,6,2,1$ elements (see, e.g., \cite {Ma}).  
The $H$-orbits in $\mathcal K$
consist of either $1$, $3$, $7$ or $21$ curves and clearly 
either $1$ or $3$ must occur in every case.  
If $H$ had a fixed curve
in $\mathcal K$, then we could blow it down to obtain a 2-dimensional
representation of $H$ which does not exist. If $H$ had an
orbit consisting of three curves, $H'$ would stabilize each of the 
curves in that orbit. Thus $H'$ would have at least six fixed points in
$Y_\mathrm{min}$ and in $Y$. This contradicts the fact that $| \mathrm{Fix}_Y(H')|=3$. 
\end {proof}
\subsection {Computation of invariants and equivalence}\label{computation}

The results of the previous section reduce the problem
of parameterizing the equivalence classes 
of K3-surfaces equipped with actions $\alpha \in \mathcal A$
to studying equivariant 2:1-covers
of surfaces $Y=X/\sigma $ branched along a single curve of genus $\geq 3$. The surface $Y$ is either $\mathbb P_2$, the blow up of $\mathbb P_2$ in seven singularities of an irreducible 
$H$-invariant sextic, or a del Pezzo
surfaces of degree two. Let us begin with a study of the first two cases where $Y_\mathrm{min} = \mathbb P_2$
\subsection*{The case where $Y_\mathrm{min} = \mathbb P_2$}
Here we will show that if $Y_\mathrm{min} = \mathbb P_2$, then 
$X$ is either a double cover of $\mathbb P_2$ branched along a 
smooth $H$-invariant sextic or the minimal desingularization of a 
double cover of $\mathbb{P}_2$ branched along a unique sextic with
seven singular points which is described below.

An effective action of $H$ on $\mathbb P_2$ is given by an injective homomorphisms $H \to \mathrm {PGL}_3(\mathbb C)$.
Due to its group structure, the only central extension of $H$ by $C_3$ is trivial.
Thus, we may regard $H$ as a subgroup of $\mathrm {SL}_3(\mathbb C)$
acting by one of its two irreducible 3-dimensional representations.
Since these representations differ by a group automorphism
and the corresponding actions on $\mathbb{P}_2$ are therefore equivalent, 
we may assume that we are only dealing with the following case:
in appropriately chosen coordinates a generator of $H'$ acts by
$[z_0:z_1:z_2]\mapsto [\lambda z_0,\lambda ^2,z_1,\lambda ^4z_2]$, 
where $\lambda =\mathrm{exp}{\frac{2\pi i}{7}}$ and a generator of
$C_3$ acts by the cyclic permutation 
$\tau $ which is defined by $[z_0:z_1:z_2]\mapsto [z_2:z_0:z_1]$. 

Any homogeneous polynomial defining an invariant curve 
must be an $H$-ei\-gen\-function with $H'$ acting with eigenvalue one.
It is a simple matter to compute the $H'$-invariant monomials:
$$ 
\mathbb C[z_0,z_1,z_2]_{(6)}^{H'}=
\mathrm {Span}\{z_0^2z_1^2z_2^2, z_0^5z_1,z_2^5z_0,z_1^5z_2\}\,.
$$
Letting $P_1=z_0^2z_1^2z_2^2 $ and $P_2=z_0^5z_1+z_2^5z_0+z_1^5z_2$,
it follows that 
$$
\mathbb C[z_0,z_1,z_2]_{(6)}^{H}=
\mathrm {Span}\{P_1,P_2\}=:V\,.
$$
There are two 1-dimensional $H$-eigenspaces, i.e., those spanned by
$z_0^5z_1+\zeta z_2^5z_0+\zeta ^2z_1^5z_2$ for $\zeta ^3=1$
but $\zeta \not =1$. By direct computation one checks that
the curves defined by these polynomials are smooth and that in both
cases all $\tau $
fixed points in $\mathbb P_2$ lie on them.
Thus, $\tau $ has only three fixed points on the K3-surface
$X$ obtained as a double cover.  But $\tau $ generates a copy of 
$C_3$ which, if it would
act by symplectic transformations,  
would have six fixed points in $X$ (\cite {N}). Consequently, 
$H$ does not lift to an action by symplectic transformations
on the K3-surfaces defined
by these curves. Hence it is enough to consider  
ramified covers $X\to Y=\mathbb P_2$, where the polynomials
defining the branch curves are invariant.

Therefore the curves under discussion are parameterized by
the 1-dimensional projective space $\mathbb P(V)$.
Our first step is to determine which polynomials 
$P_{\alpha,\beta}=\alpha P_1+\beta P_2$ define singular curves
$C=\{P=0\}$.  Clearly $P_1$ is such an example.
The location of the $\tau $-fixed points is key 
for the determination of the other singular curves.
Since $\mathrm {Fix}(\tau)$ consists of the three points 
$[1:\zeta :\zeta ^2]$, where $\zeta ^3 =1$,  the
curves which contain $\tau $-fixed points are
defined by condition $\alpha +3\zeta \beta =0$.
Let $C_{P_1}$ be the curve defined by $P_1$ and $C_\zeta $ be
that defined by $P_{\alpha,\beta}$ which satisfies
the above condition.  A direct computation shows that
the $C_\zeta $ are singular at the corresponding
$\tau $-fixed point.  So we let $\Sigma_\mathrm{reg} $ be
the complement of this set of four curves in $\mathbb P(V)$
and note the following.
\begin {proposition}
A curve $C\in \mathbb P(V)$ is smooth if and only if it is in $\Sigma_\mathrm{reg} $.
Furthermore, for every such $C$ the $H$-action on $\mathbb P_2$
lifts to a group of symplectic transformations on the K3-surface
$X$ which is defined as the 2:1 cover of $\mathbb P_2$ 
ramified over $C$.  The $H$-action is centralized by the antisymplectic involution which defines the covering
$X\to \mathbb P_2$.
\end {proposition}
\begin {proof}
First we show that
every curve $C\in \Sigma _\mathrm{reg}$ is smooth.  For this observe that,
since $\tau $ has no fixed points in $C$ and every subgroup of
order three in $H$ is conjugate to that generated by $\tau $,
$\vert H.p\vert =3,21$ for every $p\in C$.  Now the only
subgroup of order seven in $H$ is the commutator group
$H'$.  So the only possible $H$-orbits having three elements
are the orbits of the $H'$-fixed points. So we pick one such fixed
point $p$ and directly check that every $C\in \Sigma_\mathrm{reg} $ is smooth
at $p$.  Thus if $C$ is singular at some point $q$, then it
is singular at each of the 21 points in $H.q$.  By considering
the $H$-action on the space of irreducible components of $C$, one
checks that $C$ is irreducible, and therefore 
the genus formula can be applied to show that $C$ has at most 10 singularities. Hence $C$ is smooth 

Now the preimage of $H$ in $\mathrm {Aut}(X)$ is a central extension
of $H$ by $C_2$.  This splits as $H\times C_2$, where
the second factor is generated by the 2:1 covering map $\sigma $ which
acting antisymplectically. 
Since the commutator group $H'$ 
automatically acts by symplectic transformations, we must
only check that the lift of the cyclic permutation $\tau $,
$[z_0:z_1:z_2]\mapsto [z_2:z_0:z_1]$, acts symplectically.
The local
linearization of $\tau $ at the fixed point $[1:1:1] \in \mathbb P_2$ is in $\mathrm {SL}_2(\mathbb C)$.
Since this fixed point is in the complement of 
$C$, its local linearization at a corresponding fixed point in $X$ is also in $\mathrm {SL}_2(\mathbb C)$ and consequently
it is acting symplectically.
\end {proof}
Let us now turn to a description of
$\mathcal M_{\mathbb P_2}:=\mathcal A_{\mathbb P_2}/\hspace{-1.5mm}\sim $.  
Here the index $\mathbb P_2$ indicates that we have restricted
to the case where $Y=Y_{\mathrm {min}}=\mathbb P_2$.
Using the covering
$X\to \mathbb P_2$ the set $\mathcal A_{\mathbb P_2}$ becomes the
set of curves $\{h(C)\, | \,  C\in \Sigma_\mathrm{reg} , \, h\in \mathrm {PGL}_3(\mathbb C)\}$.
Equivalence is also defined by the action of $\mathrm {PGL}_3(\mathbb C)$ so that $\mathcal M_{\mathbb P_2}$
can be identified with $\Sigma_\mathrm{reg} /\hspace{-1.5mm}\sim $, where $C,\tilde C\in \Sigma_\mathrm{reg}$
are equivalent if and only if there exists $h\in \mathrm {PGL}_3(\mathbb C)$
with $h(C)=\tilde C$. Of course the normalizer $N:=N(H)$ in
$\mathrm {PGL}_3(\mathbb C)$ stabilizes $\mathbb P(V)$, because
$\mathbb P(V)$ can be viewed as the set of $H$-fixed points in 
$\mathbb P(\mathbb C[z_0,z_1,z_2]_{(6)})$.

The group $N$ is the product
$\Gamma \times H$, where $\Gamma $ is the cyclic
group of order three generated by the transformation
$[z_0:z_1:z_2]\mapsto [z_0,\zeta z_1, \zeta ^2 z_2]$,
where as above $\zeta ^3=1$.  Note that $\Gamma $
acts transitively on the set of three singular curves
$\{C_\zeta\}$ and stabilizes the curve $C_{P_1}$.
It also stabilizes the curve $C_{P_2}$ defined by the
polynomial $P_2$.  
\begin {proposition}
The equivalence relation on $\Sigma_\mathrm{reg} $ is that defined by
the $\Gamma$-action, i.e.,
$$
\mathcal M_{\mathbb P_2}=\Sigma_\mathrm{reg} /\Gamma \,.
$$
\end  {proposition}

\begin {proof}
Let $C\in \Sigma_\mathrm{reg}$ and for $T\in \mathrm {SL}_3(\mathbb C)$
assume that $T(C)\in \Sigma_\mathrm{reg} $. Consider the group span
$S$ of $THT^{-1}$ and $H$. We have shown that $H$ acts as a
group of symplectic transformations on every K3-surface defined by a curve in
$\Sigma _\mathrm{reg}$. This was proved by considering the
linearization at the $\tau $-fixed points and their location. Therefore
the same argument shows that $THT^{-1}$ also acts as a group
of symplectic transformations on the K3-surface associated to $T(C)$.
Thus $S$ is acting as a group of symplectic transformations on
this K3-surface. 

Now if $S=H$, then $T$ normalizes $H$ and, modulo $H$, is contained
in $\Gamma $.  To handle the case where $S\not=H$, note 
that $L_2(7)$ is the only group in Mukai's
list which contains a copy of $H$.  Therefore if $S\not=H$, then,
since $S$ can be realized as a subgroup of $L_2(7)$ and 
$H$ is realized as a maximal subgroup, it follows that
$S=L_2(7)$.  But there is only one curve in the family
$\mathcal M_{\mathbb P_2}$ which is $L_2(7)$-invariant,
namely $\mathrm {Hess}(C_\mathrm {Klein})$.  Thus 
$C=T(C)=\mathrm {Hess}(C_\mathrm {Klein})$ and the desired result
follows.
\end {proof}
\begin{remarks}
(1) If two K3-surfaces $X_1$ and 
$X_2$ are biholomorphically equivalent by a map 
$\varphi :X_1\to X_2$ and $X_1$ comes equipped with a $G$-action
so that it is in our family $\mathcal A_{\mathbb P_2}$, then using $\varphi $
we equip $X_2$ with an equivalent $G$-action. Thus we may
regard $X_1$ and $X_2$ as equivalent points in $\Sigma_\mathrm{reg} $. As a result
we observe that no two K3-surfaces parameterized by $\mathcal M_{\mathbb P_2}$
are biholomorphically equivalent. (2) The group $\Gamma $ stabilizes the
curve $C_{P_2}$ and therefore acts together with $H$ on the associated 
K3-surface $X$. Now the only group in Mukai's list which contains a 
copy of $H$ is $L_2(7)$ and as a subgroup $H$ is maximal.  Therefore
$\Gamma $ is acting nonsymplectically on $X$.
\end{remarks}
Our study of $H$-invariant sextic curves in $\mathbb P_2$ has revealed
the existence of precisely three singular irreducible examples
$C_\zeta$ with $\zeta^3=1$. These are identified by the action of
$\Gamma$. If $Y= X/\sigma$ is not $H$-minimal, it follows that up to
equivariant equivalence the K3-surface $X$ is the minimal
desingularization of the double cover $X_\mathrm{sing}$ of $\mathbb
P_2$ branched along $C_{\zeta=1} = C_\mathrm{sing}$, i.e., the
K3-surface arising in Corollary \ref{rough
  classi} is precisely this surface. In particular,
$Y$ is the blow up of $\mathbb P_2$ in the seven singular points of
$C_\mathrm{sing}$, which are given as the $H'$-orbit of $[1:1:1] \in
\mathbb P_2$. 
\begin{proposition}
If $Y$ is not minimal, then it is the del Pezzo surface of degree two which
arises by blowing up the seven singular points $p_1,\ldots ,p_7$ on the curve
$C_\mathrm {sing}$ in $\mathbb P_2$. The corresponding map $Y \to
\mathbb P_2$ is $H$-equivariant and therefore a Mori-reduction of
$Y$. The branch curve $B_0$ of $\pi: X \to Y$ is the proper transform
of $C_\mathrm{sing}$ in $Y$.
\end{proposition}
\begin{proof}
We need to show that the points $\{p_1, \dots, p_7\} = H'.[1:1:1]$ 
are in general position, i.e., no three lie on one line and 
no six lie on one conic.
It follows from direct computation that no three points in $H'.[1:1:1]$ 
lie on one line. 
If $p_1,\dots p_6$ lie on a conic $Q$, then $h.p_1, \dots , h.p_6$ lie on 
$h.Q$ for every  $h \in H$. Since $\{p_1, \dots, p_7\}$ is an 
$H$-invariant set, the conics $Q$ and $h.Q$ intersect in at 
least five points and therefore coincide. It follows that 
$Q$ is an invariant conic meeting $C_\mathrm{sing}$ at its 
seven singularities and $(Q,C_\mathrm{sing})=14$ implies 
$Q \subset C_\mathrm{sing}$, a contradiction.
\end{proof}
\subsection*{The role of Klein's quartic} Here we show that
the del Pezzo surface which arises as the blow up of 
the seven singular points on $C_\mathrm{sing}$ can also be
regarded as the 2:1 cover of $\mathbb P_2$ ramified over
Klein's quartic curve. For this recall that 
the anticanonical map of a del Pezzo surface $Y$ of degree two realizes it as a 2:1-cover
$Y\to \mathbb P_2$ which is ramified over a smooth curve $C$ of
degree four. This map is equivariant with respect to
the full automorphism group $\mathrm {Aut}(Y)$. Conversely,
if $C$ is any smooth curve of degree four in $\mathbb P_2$,
and $Y$ is defined as the 2:1-cover of $\mathbb P_2$ which is
ramified along $C$, then $Y$ is a del Pezzo surface of degree two.
Furthermore, the elements of the stabilizer in 
$\mathrm {Aut}(\mathbb{P}_2)$ of the curve $C$ are exactly those 
transformations which lift to automorphisms of $Y$.

Now assume as in our case that $Y$ comes equipped with an $H$-action so that
$H$ can be regarded as a subgroup of $\mathrm {PGL}_3(\mathbb C)$
which stabilizes a smooth quartic curve $C \subset \mathbb P_2$. 
In order to determine the possibilities for $C$, we choose 
coordinates so that $H$ is acting as above. It follows that
$$
\mathbb C[z_0:z_1:z_2]_{(4)}^{H'}
=\mathrm {Span}\{z_0^3z_2, z_1^3z_0,z_2^3z_1\}\,.
$$ 
This is a direct sum of $H$-eigenspaces with the eigenspace
of the eigenvalue $\zeta $ being spanned by the polynomial
$Q_\zeta :=z_0^3z_2+\zeta z_2^3z_1+\zeta ^2z_1^3z_0$ with $\zeta $
arbitrary such that $\zeta ^3=1$.

Now consider the cyclic group 
$\Gamma \subset \mathrm {GL}_3(\mathbb C)$ which 
is generated by the transformation 
$\gamma$, $(z_0,z_1,z_2)\mapsto (z_0,\zeta z_1,\zeta^2z_2)$.
The group 
$\Gamma $ acts transitively on the $H$-eigenspaces
spanned by the $Q_\zeta $.  Consequently, up to equivariant
equivalence, we may assume that
$Y\to \mathbb P_2$ is ramified over Klein's curve $C_\text{Klein}$
which is defined by $Q_1$. We therefore have the following
observation.
\begin{proposition}
A del Pezzo surface of degree two with an action of $H$ is equivariantly 
isomorphic to the double cover $Y_\mathrm{Klein}$ of $\mathbb P_2$ 
branched along Klein's quartic curve with the action of 
$H$ on $\mathbb P_2$ defined as above. 
\end{proposition}
In particular, the Mori-reduction of any $H$-action on $Y_\mathrm{Klein}$ 
is equivalent to the one defined above and in summary we have the following
result.
\begin {proposition}
If $X$ is a K3-surface equipped with an $H$-action which centralizes
an antisymplectic involution $\sigma $,  then $Y_{min}=\mathbb P_2$.
In all but one case $X/\sigma =Y=Y_{min}$. In the exceptional case
$Y=Y_\mathrm{Klein}$, the Mori-reduction $Y\to Y_{min}$ is the blow down of seven $(-1$)-curves to the singular 
points of $C_\mathrm{sing}$ and the branch set $B_0$ of $X\to Y$
is the proper transform of $C_\mathrm {sing}$ in $Y$.
\end {proposition}
\begin {proof}
It remains to prove that $B_0$
is the proper transform of $C_\mathrm {sing}$ in $Y$.  For this
suppose that the
branch curve of $X\to Y$ were some other curve 
$\widetilde B_0$ in the linear system of $-2K_Y$. For 
$I:=\widetilde B_0\cap B_0$ we note that $\vert I\vert \le 8$.
But since $H$ has no fixed points in $B_0$, it follows that
$\vert I\vert =3$ and that $I$ is an $H$-orbit.  Thus the
intersection multiplicities of $\widetilde B_0\cap B_0$ are
the same along $I$, contrary to the fact that 3 does not divide
8.
\end {proof}
\subsection*{Completion of the proof of Theorem \ref{main theorem}}
To complete the proof of Theorem \ref{main theorem} we
must first show that if $H$ is acting on $\mathbb P_2$ as
above, then it lifts to a group acting on the K3-surface $X$ 
which is defined as the 2:1 cover of the del Pezzo surface $Y_\text{Klein}$
which in turn is defined as the cover of $\mathbb P_2$ ramified
over $C_\text{Klein}$.  Of course we mean that $X\to Y_\text{Klein}$ is ramified over
the preimage $B_0$ of $C_\text{Klein}$. Since $H$ stabilizes $C_\text{Klein}$ and does not admit nontrivial central extensions of degree two, it lifts to a subgroup of $\mathrm{Aut}(Y_\text{Klein})$. By the same argument $H$ lifts to a subgroup of $\mathrm{Aut}(X)$.

Secondly, the covering involution  
$Y_\text{Klein}\to \mathbb P_2$,
lifts to a holomorphic transformation of $X$. On $X$ we also consider the
involution defining $X \to Y_\text{Klein}$. Together these transformations
generate a group
of order four, every element of which has a positive-dimensional
fixed point set. Therefore this group acts as $ C_4$
by a character on any choice of the symplectic form of $X$.
Thus the full preimage of $H$ in $\mathrm {Aut}(X)$ splits uniquely
as a product $H\times C_4$.  Since the commutator group $H'$ 
automatically acts by symplectic transformations, we must
only check that the lift of the cyclic permutation $\tau $,
$[z_0:z_1:z_2]\mapsto [z_2:z_0:z_1]$, acts symplectically.
As before, this follows from a linearization argument at a $\tau$-fixed point not in $C_\mathrm{Klein}$. 

Thus we observe that in this situation up to equivalence there
is a unique action of $H$ by symplectic transformations 
on a unique K3-surface $X_{KM}$, and this is centralized by a cyclic 
group of order four which acts faithfully as $C_4$
by a multiplicative character on any choice of the symplectic
form. 

It follows that $\mathcal M \backslash \mathcal M_{\mathbb P_2} =
\{X_\mathrm{KM}\}$, i.e., the Klein-Mukai-surface is the only surface in the family
$\mathcal M$ for which $Y \not\cong \mathbb P_2$. If we define $\Sigma$ as the complement of $C_{P_1}$ in $\mathbb P(V)$. Then $\Sigma = \Sigma_\mathrm{reg} \cup \{C_\zeta \, | \, \zeta^3 =1\}$ and 
\[
\mathcal M = \Sigma / \Gamma.
\]
The above discussion completes the proof of Theorem \ref{main theorem}. 
It remains to determine which K3-surfaces in the family 
$\mathcal{M}$ have $L_2(7)$-symmetry.
 \subsection{The $L_2(7)$-action}
First observe that an action of $L_2(7)$ on $\mathbb{P}_2$ is necessarily 
given by one of its two 3-dimensional irreducible representations, 
which differ by an outer automorphism of the group. We may therefore 
consider one particular representation such that the subgroup $H$ is 
represented as above and check that the curve $\mathrm{Hess}(C_\text{Klein}) \in \Sigma_\mathrm{reg}$ 
is $L_2(7)$-invariant. It is in fact the unique curve in $\Sigma_\mathrm{reg}$ with this 
property. This is a well-known result from the invariant theory of the group 
$L_2(7)$ but can also be seen as follows: suppose there were two distinct 
invariant smooth sextic curves $C_s,C_s'$. The maximal possible isotropy of 
$L_2(7)$ is the cyclic group of order 7 so that each $L_2(7)$-orbit on 
$C_s$ has at least 21 elements. It follows that there is no configuration 
of $C_s$ and $C_s'$ which fulfills $C_s.C_s'=36$, a contradiction. 

We have hereby singled out a unique K3-surface with $L_2(7)$-symmetry in 
the family $\mathcal{M}_{\mathbb P_2}$. Since $L_2(7)$ is a simple group and
in particular is equal to its commutator, its action on this 
surface is clearly symplectic and by construction is centralized by the antisymplectic 
covering involution.

To complete the proof of Theorem \ref{thmL2(7)}, note that the curve
$C_\text{Klein}\subset \mathbb{P}_2$ is $L_2(7)$-invariant with
respect to the representation discussed above. We see that $L_2(7)$ lifts to
a subgroup of $\mathrm {Aut}(X_{KM})$. 
Hence, analogous to the case of $H$, the preimage
of $L_2(7)$ in $\mathrm {Aut}(X_{KM})$ is $L_2(7)\times  C_4$,
where $L_2(7)$ is acting symplectically and $ C_4$ is
acting by holomorphic nonsymplectic transformations.  The
generator $\sigma $ of the subgroup isomorphic to $C_2$ in
$C_4$ is the antisymplectic involution which
centralizes $L_2(7)$ in the above discussion. 

In conclusion we reiterate that both of the K3-surfaces which
have an $L_2(7)$-action centralized by an involution appear in
the family parameterized by $\mathcal M$. For the K3-surface
constructed as 
the 2:1 cover of $\mathbb P_2$ ramified along $\mathrm {Hess}(C_\mathrm{Klein})$
we clearly have $Y=Y_\mathrm{min}=\mathbb P_2$ for both $H$ and $L_2(7)$.
For $X_\mathrm{KM}$ the quotient $Y=Y_\mathrm{Klein}$ is an
$L_2(7)$-minimal model, whereas
$\mathbb P_2$ is the minimal model with respect to the action of $H$.
%
%
%
%

\end{document}